\documentclass{article}

\usepackage{amsmath}
\usepackage{amssymb}
\usepackage{graphicx} 
\usepackage{amsthm}

\newtheorem{theorem}{Theorem}

\newtheorem{lemma}[theorem]{Lemma}
\newtheorem{corollary}[theorem]{Corollary}

\def\lv #1{\mathbb #1}

\begin{document}
\title{Torsion of Elliptic Curves Over Quadratic Fields}         
\author{Jody Ryker, Sophie De Arment}        
\date{\today}

\maketitle

\begin{center}
\textbf{Abstract}
\end{center}

By focusing on the family $E:y^2=x^3+a$, we present strategies for determining the structure of the torsion
subgroup of the Mordell-Weil group of an elliptic curve, $E(K)$, over
quadratic field $K$. Generalizations of the Nagell-Lutz
theorem and Mazur's theorem to curves defined over quadratic fields allows us to determine the full torsion subgroup of
$E(K)$ as one of at most three possibilities when $a$ is a square.

\section{Introduction}

The structure of the torsion subgroup of an elliptic curve over $\lv{Q}$ is well understood. Mordell's theorem states that the set of rational torsion points on $E$ is a finitely-generated abelian group. As a result, there are finitely many points of rational torsion on $E$. Further, Mazur's Theorem describes the possible structures of $E(\lv{Q})_{tors}$. Finally, we can use Nagell-Lutz's theorem to compute all the rational torsion points of a given elliptic curve \cite{st}.
So, the Mordell, Mazur, and Nagell-Lutz theorems provide a complete description of the torsion subgroup of any elliptic curve over $\lv{Q}$. We would like to have a similar description for the torsion subgroup of elliptic curves over quadratic fields.
An extension of Mordell's theorem, shows that $E(K)_{tors}$, where $K$ is a quadratic field, is also a finitely generated abelian group \cite{st}. The following theorem of Kamienny, Kenku, and Momose lists the 26 possibilities for the structure of $E(K)_{tors}$.
\begin{theorem}\label{kkm} (Kamienny \cite{k}, Kenku and Momose \cite{km})
Let $K$ be a quadratic field and $E$ an elliptic curve over $K$. Then the torsion subgroup $E(K)_{tors}$ of $E(K)$ is isomorphic to one of the following 26 groups:
$$\mathbb{Z} / m\mathbb{Z}, \text{ for } 1 \leq m \leq18, m \neq 17,$$
$$\mathbb{Z} / 2\mathbb{Z} \oplus \mathbb{Z} / 2m\mathbb{Z}, \text{ for } 1 \leq m \leq 6,$$
$$\mathbb{Z} / 3\mathbb{Z} \oplus \mathbb{Z} / 2m\mathbb{Z}, \text{ for } m=1,2$$
$$\mathbb{Z} / 4\mathbb{Z} \oplus \mathbb{Z} / 4\mathbb{Z}.$$
\end{theorem}

In this paper, we concentrate on more specifically characterizing $E(K)_{tors}$ for families of curves. In particular, we can pare down the list of 26 possibilities to at most three for curves of the form $y^2=x^3+a$, where $a$ is a square.
We also present strategies for describing $E(K)_{tors}$ for other families of elliptic curves. We generalize Nagell-Lutz's theorem to determine where 2-torsion occurs. We also describe a method for finding 3-torsion. Using this information, we can more specifically describe the possibilities for $E(K)_{tors}$. Next, we compare curves with parameterizations given in Rabarison's work (\cite{fpr}) for curves having certain torsion structures. Finally, we consider torsion structures that we know only occur over quadratic cyclotomic fields (\cite{fn}, \cite{fr}).

In Section \ref{smalltors}, we will present methods for finding 2-torsion and 3-torsin points. In Section \ref{proof}, we will prove our main result:
\begin{theorem}\label{result}
Let $E(K):y^2=x^3+a,$ where $a$ is an integer and $K$ is any quadratic field.

\renewcommand{\theenumi}{\roman{enumi}}
\begin{enumerate}
\item Suppose $a$ is a sixth power integer.

\noindent If $K\neq \lv{Q}(\sqrt{-3})$, then $E(K)_{tors}$ is isomorphic to $ \lv{Z}/6\lv{Z}, \lv{Z}/12\lv{Z},$ or $\lv{Z}/18\lv{Z}$.
\newline\noindent If $K= \lv{Q}(\sqrt{-3})$, then $E(K)_{tors}$ is isomorphic to $\lv{Z}/2\lv{Z}\oplus\lv{Z}/6\lv{Z}$.
\item Suppose $a$ is a square but not a sixth power, and $K\neq\lv{Q}(\sqrt{-3})$. Then $E(K)_{tors}$ is isomorphic to $\lv{Z}/3\lv{Z}$ or $\lv{Z}/9\lv{Z}$.
\newline\noindent If $K=\lv{Q}(\sqrt{-3})$, then $E(K)_{tors}$ is isomorphic to $\lv{Z}/3\lv{Z},\lv{Z}/9\lv{Z},$ or $\lv{Z}/3\lv{Z}\oplus\lv{Z}/3\lv{Z}$.

\end{enumerate}
\end{theorem}

We will also describe the torsion structure of a particular curve using Theorem \ref{result}.

\begin{corollary}
Let $E(K):y^2=x^3+1,$ where $K$ is any quadratic field. Then $E(K)_{tors}$ is isomorphic to $\lv{Z}/6\lv{Z}$, $\lv{Z}/18\lv{Z}$, or $\lv{Z}/2\lv{Z}\oplus\lv{Z}/6\lv{Z}$.
\end{corollary}

\section{2-Torsion and 3-Torsion}\label{smalltors}

Let $E(K): \, y^2 =x^3+bx+a$ be an elliptic curve with $a, b \in K$ with $K$ quadratic. We will first consider over which quadratic fields 2-torsion and 3-torsion occur in order to more precisely describe $E(K)_{tors}$.
\begin{lemma} \label{2tors}
A non-trivial point $(x,y)$ on a curve $E(K): y^2=x^3+bx+a$, where $a, b \in K$, is a point of order two if and only if $x\in K$ satisfies $x^3+bx+a=0$.
\end{lemma}

\begin{proof}
Proof: From Nagell-Lutz's theorem we know that a point $(x,y)\neq \mathcal{O}$ on $E$ is a point of order two if and only if $y=0$ \cite{st}.
\begin{flushright}
 $\blacksquare$
\end{flushright}

\end{proof}

 In other words, we factor $x^3+bx+a$ to identify the fields over which $E$ has 2-torsion. In order for $E(K)_{tors}$ to contain a 2-torsion point, there would necessarily be an element $x\in K$ such that $x$ satisfiess $x^3+bx+a=0$.

\begin{lemma}\label{3tors} \noindent\renewcommand{\theenumi}{\roman{enumi}}
\begin{enumerate}
\item Let $E(K): y^2=x^3+bx+a$, where $a, b \in K$. A point $(x,y) \in E(K)$ is a point of order three if and only if $x\in K$ is satisfies $3x^4+6bx^2+12ax-b^2=0$.
\item Let $E(K): y^2=x^3+a$, where $a\in K$. A point $(x,y) \in E(K)$ is a point of order three if and only if $x\in K$ satisfies $3x^4+12ax=0$. When $a$ is a square, there will always be a point of order three on $E(\lv{Q})$.
\end{enumerate}
\end{lemma}

\begin{proof}
\renewcommand{\theenumi}{\roman{enumi}}
Proof: \begin{enumerate}
\item First recall that points of order three are points of inflection of $E$ \cite{st}. We take the second derivative of $E(K)$, and find that the $x$-coordinate a point of order three must be a root of 
\begin{equation}
3x^4+6bx^2+12ax-b^2=0.
\end{equation}

\item If $b=0$, Equation (1)
simplifies to 
\begin{equation}
3x^4+12ax=0.
\end{equation}
The inflection points occur at 
\[x=0,\, x=-\sqrt[3]{4a},\, x=-\frac{\sqrt[3]{4a}(1-\sqrt{3})}{2}, \text{ and } x=-\frac{\sqrt[3]{4a}(1+\sqrt{3})}{2}. \]

Since $x=0$ results in a 3-torsion point, there will be a point of order three, namely $(0,\pm\sqrt{a})$. If $a$ is a square, this point is in $\lv{Q}$. 
\begin{flushright}
$\blacksquare$
\end{flushright}

\end{enumerate}
\end{proof}


\section{Torsion Over Quadratic Fields}\label{proof}
\setcounter{theorem}{1}
\begin{theorem}
Let $E(K):y^2=x^3+a,$ where $a$ is an integer and $K$ is any quadratic field.

\renewcommand{\theenumi}{\roman{enumi}}
\begin{enumerate}
\item Suppose $a$ is a sixth power integer.

\noindent If $K\neq \lv{Q}(\sqrt{-3})$, then $E(K)_{tors}$ is isomorphic to $ \lv{Z}/6\lv{Z}, \lv{Z}/12\lv{Z},$ or $\lv{Z}/18\lv{Z}$.
\newline\noindent If $K= \lv{Q}(\sqrt{-3})$, then $E(K)_{tors}$ is isomorphic to $\lv{Z}/2\lv{Z}\oplus\lv{Z}/6\lv{Z}$.
\item Suppose $a$ is a square but not a sixth power, and $K\neq\lv{Q}(\sqrt{-3})$. Then $E(K)_{tors}$ is isomorphic to $\lv{Z}/3\lv{Z}$ or $\lv{Z}/9\lv{Z}$.
\newline\noindent If $K=\lv{Q}(\sqrt{-3})$, then $E(K)_{tors}$ is isomorphic to $\lv{Z}/3\lv{Z},\lv{Z}/9\lv{Z},$ or $\lv{Z}/3\lv{Z}\oplus\lv{Z}/3\lv{Z}$.

\end{enumerate}
\end{theorem}

\begin{proof} Proof: \noindent\renewcommand{\theenumi}{\roman{enumi}} \begin{enumerate} \item Let $E(K)$: $y^2=x^3+b^6$, where $a=b^6$ and $b\in\lv{Z}$. Since 2-torsion occurs when $y=0$, we will look for the roots of $x^3+b^6$  (Lemma \ref{2tors}). These are
$$-b^{6/3}=-b^2$$
and
$$\frac{b^{6/3}\pm b^{6/3} \sqrt{-3}}{2}=\frac{b^2 \pm b^2 \sqrt{-3}}{2}.$$
There is exactly one rational root, $b^2$, so there is one point of order two over $\lv{Q}$ and at least one point of order two over every quadratic extension of $\lv{Q}$. There are two more points of order two over $\lv{Q}(\sqrt{-3})$.

Since there will be at least one rational point of order three, $\lv{Z}/3\lv{Z}\subseteq E(\lv{Q})_{tors}$ (Lemma \ref{3tors}). Further, since we also know that $\lv{Z}/2\lv{Z}\subseteq E(\lv{Q})_{tors}$, then $\lv{Z}/6\lv{Z}\subseteq E(\lv{Q})_{tors}$.

Let $K$ be any quadratic field. Then since $\lv{Z}/6\lv{Z}\subseteq E(\lv{Q})_{tors}$, $\lv{Z}/6\lv{Z}\subseteq E(K)_{tors}$. Using Kamienny, Kenku, and Momose's work (\cite{k}, \cite{km}) in Theorem \ref{kkm}, $E(K)_{tors}$ is one of the following:
$$ \lv{Z}/6\lv{Z}$$
$$\lv{Z}/12\lv{Z}$$
$$\lv{Z}/18\lv{Z}$$
$$\mathbb{Z} / 2\mathbb{Z} \oplus \mathbb{Z} / 6\mathbb{Z}$$
$$\mathbb{Z} / 2\mathbb{Z} \oplus \mathbb{Z} / 12\mathbb{Z}$$
$$\mathbb{Z} / 3\mathbb{Z} \oplus \mathbb{Z} / 6\mathbb{Z}$$

If $K\neq \lv{Q}(\sqrt{-3})$, there are no further points of order two. Also, $\lv{Z}/3 \lv{Z} \oplus \lv{Z}/ 6 \lv{Z}$ does not occur over any quadratic field other than $\lv{Q}(\sqrt{-3})$ \cite{fn}. Hence, this limits the above list to $ \lv{Z}/6\lv{Z}, \lv{Z}/12\lv{Z}$, or $\lv{Z}/18\lv{Z}$.

If $K=\lv{Q}(\sqrt{-3})$, there will be three non-trivial points of order two (Lemma \ref{2tors}). Due to Najman's work, it has already been shown that $E(\lv{Q}(\sqrt{-3}))_{tors}$ cannot be $\lv{Z} / 2\mathbb{Z} \oplus \mathbb{Z} / 12\mathbb{Z}$ \cite{fn}. This leaves one possibility for the torsion subgroup of $E(\lv{Q}(\sqrt{-3}))$: $\mathbb{Z} / 2\mathbb{Z} \oplus \mathbb{Z} / 6\mathbb{Z}$.

\item Now let $a=b^2$, where $b$ is not a cube. We already know that there will be one point of order three (Lemma \ref{3tors}). Hence, $\lv{Z}/3 \lv{Z} \subseteq E(\lv{Q})_{tors} \subseteq E(K)_{tors}$. This shortens the list of 26 possibilities in Theorem \ref{kkm} down to nine (\cite{k},\cite{km}):
$$ \lv{Z}/6\lv{Z}$$
$$\lv{Z}/9\lv{Z}$$
$$\lv{Z}/12\lv{Z}$$
$$ \lv{Z}/15\lv{Z}$$
$$\lv{Z}/18\lv{Z}$$
$$\mathbb{Z} / 2\mathbb{Z} \oplus \mathbb{Z} / 6\mathbb{Z}$$
$$\mathbb{Z} / 2\mathbb{Z} \oplus \mathbb{Z} / 12\mathbb{Z}$$
$$\mathbb{Z} / 3\mathbb{Z} \oplus \mathbb{Z} / 3\mathbb{Z}$$
$$\mathbb{Z} / 3\mathbb{Z} \oplus \mathbb{Z} / 6\mathbb{Z}$$

Next, we consider if and when 2-torsion will occur by finding the roots of $y^2=x^3+b^2$ (Lemma \ref{2tors}). The roots are:
$$-b^{2/3}$$
and 
$$\frac{b^{2/3} \pm b^{2/3} \sqrt{-3}}{2}.$$

Since $b$ is not a cube, these roots are not contained in any quadratic field, thus eliminating the possibility of 2-torsion points on $E(K)$. This now shortens the list for $E(K)_{tors}$ to:

$$\lv{Z}/3\lv{Z}$$
$$\lv{Z}/9\lv{Z}$$
$$ \lv{Z}/15\lv{Z}$$
$$\mathbb{Z} / 3\mathbb{Z} \oplus \mathbb{Z} / 3\mathbb{Z}$$







Finally, we know that $\mathbb{Z} / 3\mathbb{Z} \oplus \mathbb{Z} / 3\mathbb{Z}$ torsion does not occur over any quadratic field other than $\lv{Q}(\sqrt{-3})$ (\cite{fn}), and $\lv{Z}/15\lv{Z}$ torsion never occurs over $\lv{Q}(\sqrt{-3})$ (\cite{fr}). 
\begin{flushright}
$\blacksquare$
\end{flushright}

\end{enumerate}

\end{proof}

Below are examples of curves of the form $y^2=x^3+a$.
$$E: y^2=x^3+1\text{ , }E(\lv{Q}(\sqrt{-3}))\cong\lv{Z}/2\lv{Z}\oplus\lv{Z}/6\lv{Z}$$
$$E: y^2=x^3+4^2\text{ , }E(\lv{Q}(\sqrt{-3}))\cong\lv{Z}/3\lv{Z}\oplus\lv{Z}/3\lv{Z}$$

\begin{corollary}
Let $E(K):y^2=x^3+1,$ where $K$ is any quadratic field. Then $E(K)_{tors}$ is isomorphic to $\lv{Z}/6\lv{Z}$, $\lv{Z}/18\lv{Z}$, or $\lv{Z}/2\lv{Z}\oplus\lv{Z}/6\lv{Z}$.
\end{corollary}
Proof:
From Theorem \ref{result}, we know that $E(K)_{tors}$ is isomorphic to $\lv{Z}/6\lv{Z}$, $\lv{Z}/12\lv{Z}$, $\lv{Z}/18\lv{Z}$, or $\lv{Z}/2\lv{Z}\oplus\lv{Z}/6\lv{Z}$. We will compare the curve $y^2=x^3+1$ with the parameterization given for a curve $E(K)$ with $\lv{Z}/12\lv{Z}\subseteq E(K)_{tors}$ in Rabarison's paper \cite{fpr}. First, we convert $y^2=x^3+1$ from Weierstrass form to Tate Normal form:
$$y^2+\frac{4}{3}xy+\frac{2}{9}y=x^3+\frac{2}{9}x^2$$

Below is the parameterization of \cite{fpr}:

$$y^2 + (6t^4-8t^3+2t^2+2t-1)xy +(-12t^6+30t^5-34t^4+21t^3-7t^2+t)(t-1)^5y =$$$$= x^3 +(-12t^6+30t^5-34t^4+21t^3-7t^2+t)(t-1)^2x^2$$

The parameterization is centered at a point of order 12. We transform the curve so it is centered at a point of order six, as $y^2+\frac{4}{3}xy+\frac{2}{9}y=x^3+\frac{2}{9}x^2$ is. Next, we compare coefficients:
$$-\frac{-10t^4+20t^3-16t^2+6t-1}{(3t^2-3t+1)^2}=\frac{4}{3}$$
$$\frac{t^2(t-1)^2(2t-1)^2(2t^2-2t+1)}{(3t^2-3t+1)^4}=\frac{2}{9}$$

This system is inconsistent. Hence, $\lv{Z}/12\lv{Z}\not\cong E(K)_{tors}$.

\begin{flushright}
$\blacksquare$
\end{flushright}

We believe that $E(K)_{tors}\not\cong\lv{Z}/18\lv{Z}$ for all $K$. However, we were unable to prove this. We attempted to compare the curve $y^2=x^3+1$ to parameterizations for a curve with $\lv{Z}/18\lv{Z}\subseteq E(K)_{tors}$. We obtained a system of equations that Mathematica was unable to solve. We checked the torsion structure of $E(\lv{Q}(\sqrt{d}))$ for $-9000\leq d\leq 3814$, and determined that $E(\lv{Q}(\sqrt{d}))\not\cong\lv{Z}/18\lv{Z}$ for such $d$.

%
%
%
%
%
%
%

\bibliography{EllipticBib}
\bibliographystyle{plain}

\end{document}